\documentclass{article}
\pdfoutput=1
\usepackage{amsmath}
\usepackage{amscd}
\usepackage{amssymb}
\usepackage{amsfonts}
\usepackage{amsthm}
\usepackage{graphicx}
\usepackage{cite}
\usepackage{hyperref}
\newcommand{\sphere}{\mathbb{S}}
\newcommand{\torus}{\mathbb{T}}
\newcommand{\real}{\mathbb{R}}
\newcommand{\integer}{\mathbb{Z}}

\newcommand{\set}[1]{\{#1\}}

\newtheorem{thm}{Theorem}

\newtheorem{prop}{Proposition}[section]
\newtheorem{lem}{Lemma}[section]

\theoremstyle{remark}
\newtheorem*{rem}{Remark}

\theoremstyle{definition}

\date{December 21, 2012}
\begin{document}
\author{ Haomin Wen\thanks{Department of Mathematics, University of Pennsylvania, Philadelphia, PA 19104, USA e-mail: weh@math.upenn.edu}}
\title{Lipschitz minimality of the multiplication maps of unit complex, quaternion and octonion numbers}
\maketitle

\begin{abstract}
We prove that the multiplication maps
$\sphere^n \times \sphere^n \rightarrow \sphere^n$ ($n = 1, 3, 7$)
for unit complex, quaternion and octonion numbers are,
up to isometries of domain and range,
the unique Lipschitz constant minimizers in their homotopy classes.
Other geometrically natural maps, such as projections of Hopf fibrations,
have already been shown to be, up to isometries,
the unique Lipschitz constant minimizers in their homotopy classes,
and it is suspected that this may hold true for all Riemannian submersions of compact homogeneous spaces.
\end{abstract}

\section[Introduction]{Introduction}
\subsection{Background}
A map $f : M \rightarrow N$ between two metric spaces $(M, d_M)$ and $(N, d_N)$
is a \emph{Lipschitz map}
if there is $C > 0$ such that $d_N(f(x_1), f(x_2)) \le C \  d_M(x_1, x_2)$ for any $x_1, x_2 \in M$.  
The smallest such constant $C$ is called the \emph{Lipschitz constant} of $f$
and denoted by $L_f(M)$.
If a Lipschitz map has the smallest Lipschitz constant in its homotopy class,
then it is called a \emph{Lipschitz constant minimizer}.
Note that
there always exists a Lipschitz constant minimizer (by Arzel\`a-Ascoli)
in the homotopy class of any Lipschitz map from $M$ to $N$
when $M$ and $N$ are compact.

Sometimes it is possible to recognize certain special maps in terms of
Lipschitz constant and homotopy class.
Previously, there have been results using other invariants
like volume or energy,
but even some of the simplest maps can not be characterized by just using these two invariants.
For example,
the inclusion map $\Delta \sphere^3 \rightarrow \sphere^3 \times \sphere^3$
is neither volume minimizing (since $\sphere^3 \vee \sphere^3$ has smaller volume) nor energy minimizing \cite{white-1986} in its homotopy class.
However, it is shown \cite{deturk-gluck-storm-2010} that
this map is the Lipschitz constant minimizer in its homotopy class,
unique up to isometries on the domain and range.
See \cite{deturk-gluck-storm-2010} for more examples of this type including Hopf fibrations.


\subsection{Main result}
The authors of \cite{deturk-gluck-storm-2010} suspected that many more maps,
such as Riemannian submersions of compact homogeneous spaces,
are Lipschitz constant minimizers in their homotopy classes,
unique up to isometries on the domain and range.
(It is necessary to assume certain homogeneity,
otherwise there will be counterexamples as shown in Section \ref{counter_example}.)
Then it is natural to consider group multiplication maps on compact groups which provide an easy class of Riemannian submersions of compact homogeneous spaces.
The simplest case is $\sphere^1 \times \sphere^1 \rightarrow \sphere^1$,
which is more or less trivial.
The first interesting compact group to look at is $\sphere^3$ and we have the following theorem.

\begin{thm}
    The Lipschitz constant of any map $f : \sphere^n \times \sphere^n \rightarrow \sphere^n$ ($n = 1, 3, 7$) homotopic to the multiplication map $m$ of unit complex, quaternion or octonion numbers is $ \ge \sqrt{2}$,
    with equality if and only if $f$ is isometric to $m$.
    \label{dim3}
\end{thm}

$f_1 : M \rightarrow N$ and $f_2 : M \rightarrow N$ are said to be \emph{isometric},
if there are isometries $g_M : M \rightarrow M$ and $g_N : N \rightarrow N$
such that $g_N \circ f_1 = f_2 \circ g_M$.

\begin{rem}
    The multiplication map of $\sphere^1$ is an energy minimizer in its homotopy class,
    but the multiplication map of $\sphere^3$ is not.
    In fact,
    by a result in \cite{white-1986},
    the energy of maps homotopic to the identity map on $\sphere^3$ can be arbitrarily small.
    If we construct a map $f : \sphere^3 \rightarrow \sphere^3$ as in \cite{white-1986}
    which is homotopic to the identity map
    and which is of very small energy,
    then a direct computation will show that
    the composition of $f$ and the multiplication map
    is also of very small energy.
\end{rem}

\subsection{Acknowledgements}
The author is especially grateful to Herman Gluck,
for his guidance and encouragement during the preparation of the paper.
Many ideas in this paper are from discussions with him and Figure \ref{fig_f3d}
was drawn by him.
Also,
the idea of using Proposition \ref{prop_ext} is due to Herman.

Many thanks to Christopher Croke
for carefully reading the draft of this paper,
to Clayton Shonkwiler for
an earlier proof of Proposition \ref{prop_at_least},
and to Haggai Nuchi for
explaining Proposition \ref{prop_ext} to me in detail.

Special thanks to
Dennis DeTurck,
Paul Melvin,
Shea Vela-Vick,
Kerstin Baer,
Li-Ping Mo
and
Eric Korman,
for their help during the preparation of the paper.

%

%
%
%


\section{Proof of Theorem \ref{dim3} }
In this section,
$m : \sphere^n \times \sphere^n \rightarrow \sphere^n$ ($n = 1, 3, 7$)
will denote the multiplication map of unit complex, quaternion or octonion numbers and
$f$ will be a map homotopic to $m$.
On any Riemannian manifold, $d$ will denote the distance function generated by the underlying Riemannian metric.
\subsection{Lipschitz minimality}

The Lipschitz minimality follows from the following theorem from \cite{gromov-1978},
which was first proved in \cite{olivier-1966} when $d$ is even.
There is also a proof of this theorem in \cite{gromov-2007}.
\begin{prop}[\cite{olivier-1966, gromov-1978}]
    Suppose $g : \sphere^n \rightarrow \sphere^n$ is of degree $d$.
    When $|d| \ge 2$,
    $L_g(\sphere^n) \ge 2$.
    \label{prop_o}
\end{prop}

We shall use this result with $n = 1, 3, 7$ and $d = 2$ to prove the following proposition.

\begin{prop}
    The Lipschitz constant of $f$ is at least $\sqrt{2}$,
    that is to say,
    $L_f(\sphere^n \times \sphere^n) \ge \sqrt{2}$.
    \label{prop_at_least}
\end{prop}

\begin{proof}
    Consider the restriction of $f$ on
    $\Delta\sphere^n = \{(x, x) : x \in \sphere^n \}$,
    the diagonal sphere.
    Since $m |_{\Delta\sphere^n}$ is of degree 2
    ($m(x, x) = x^2$)
    and since $f$ is homotopic to $m$,
    $f |_{\Delta\sphere^n}$ is a degree 2 map from $\Delta\sphere^n$ (isometric to $\sqrt{2} \sphere^n$) to $\sphere^n$.
    By Proposition \ref{prop_o}, $L_f(\Delta\sphere^n) \ge \sqrt{2}$,
and hence $L_f(\sphere^n \times \sphere^n) \ge \sqrt{2}$.
\end{proof}

\begin{prop}
    $L_m(\sphere^n \times \sphere^n) = \sqrt{2}$,
    that is, 
    the Lipschitz constant of $m$ is $\sqrt{2}$.
    \label{prop_upper_bound}
\end{prop}

\begin{proof}
    For any $(x, y) \in \sphere^n \times \sphere^n$,
    we have an orthogonal decomposition
    \begin{align*}
    T_{(x, y)} \left(\sphere^n \times \sphere^n\right)= T_x \left(\sphere^n \times \set{y}\right) \oplus T_y \left(\set{x} \times \sphere^n \right).
    \end{align*}

    Since $m|_{ \sphere^n \times \set{y} } $is an isometry, 
    $dm|_{T_x \left(\sphere^n \times \set{y} \right)}$ is also an isometry.
    Similarly,
    $dm|_{T_y \left(\set{x} \times \sphere^n\right)}$ is an isometry.

    For any $X \in T_x \left(\sphere^n \times \set{y}\right)$ and
    $Y \in T_y \left(\set{x} \times \sphere^n\right)$,
    \begin{align*}
	\left| dm(X + Y) \right|
	&\le |dm(X)| + |dm(Y)|\\
	&= |X| + |Y|\\
	&\le \sqrt{2} \sqrt{|X|^2 + |Y|^2}\\
	&= \sqrt{2} |X + Y|.
    \end{align*}
    Hence $L_m(\sphere^n \times \sphere^n) \le \sqrt{2}$.
    By Proposition \ref{prop_at_least},
    we also have $L_m(\sphere^n \times \sphere^n) \ge \sqrt{2}$,
    and hence $L_m(\sphere^n \times \sphere^n) = \sqrt{2}$.
\end{proof}

\subsection{Uniqueness: plan of the proof}
Now $m$ is a Lipschitz constant minimizer in its homotopy class,
and it remains to show the uniqueness.

In this section,
suppose that $f$ is also of Lipschitz constant $\sqrt{2}$,
then we need to prove that $f$ and $m$ are isometric.

Our plan is as following.
\begin{enumerate}
    \item Show that the fibers of $f$ are
	parallel spheres isometric to $\sqrt{2} \sphere^n$.
	\label{step_plan_uniq_1}
    \item Use this result to prove that $f$ is isometric to $m$.
\end{enumerate}

\subsection{Basic tools}
The main tool in the second step is the theory on isoclinic planes and Cliffold algbra,
and the main tool in the first step will be the following inequalities.
\begin{prop}
    Suppose $f : \sphere^n \times \sphere^n \rightarrow \sphere^n$
    is homotopic to $m$ and of Lipschitz constant $\sqrt{2}$
    and suppose $p \in \sphere^n$,
    then for any $(x_1, y_1)$ and $(x_2, y_2) \in f^{-1}(p)$,
    we have
    \begin{align}
	\left(2 \pi - d(x_1, x_2)\right)^2 + \left( d(y_1, y_2) \right)^2 \ge 2 \pi^2,
	\label{eq_neq1}
    \end{align}
    and
    \begin{align}
	\left(d(x_1, x_2)\right)^2 + \left(2 \pi - d(y_1, y_2) \right)^2 \ge 2 \pi^2.
	\label{eq_neq2}
    \end{align}
    \label{prop_neq}
\end{prop}
\begin{figure}[h]
    \centering
    \includegraphics{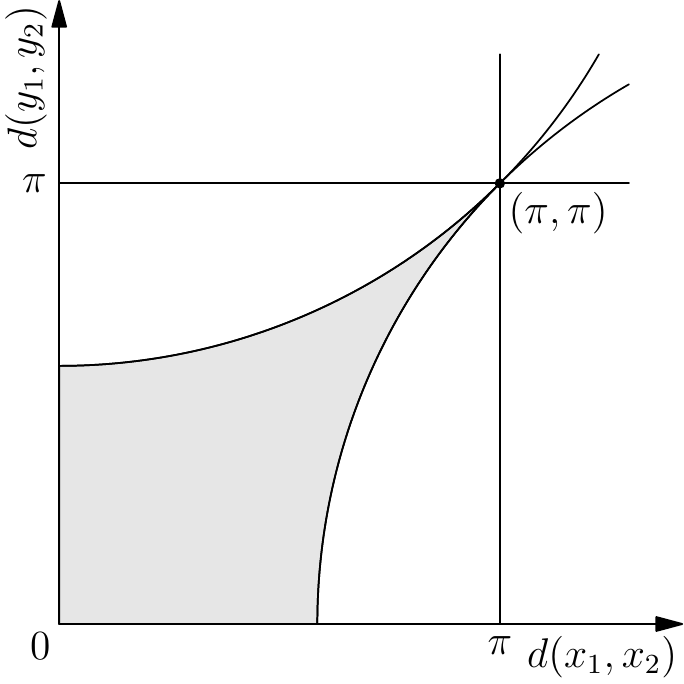}
    \caption{$(x_1, y_1)$ and $(x_2, y_2)$ are in a same fiber of $f$ only if $(d(x_1, x_2), d(y_1, y_2))$ is in the shaded region.}
\label{fig_neq}
\end{figure}
\begin{rem}
    We can see Proposition \ref{prop_neq} in Figure \ref{fig_neq}.
    The shaded region in Figure \ref{fig_neq}
    is the set of points $(d(x_1, x_2), d(y_1, y_2))$ 
    satisfying \eqref{eq_neq1}, \eqref{eq_neq2},
    $0 \le d(x_1, x_2) \le \pi$ and $0 \le d(y_1, y_2) \le \pi$.
    In other words,
    $(x_1, y_1)$ and $(x_2, y_2)$ are in a same fiber of $f$ only if $(d(x_1, x_2), d(y_1, y_2))$ is in the shaded region.

    If $x_1$ and $x_2$ are antipodal points,
    that is,
    $d(x_1, x_2) = \pi$,
    then we can see from the graph instantly that
    $y_1$ and $y_2$ are also antipodal points,
    that is,
    $d(y_1, y_2) = \pi$.
    Moreover,
    $(\pi - d(x_1, x_2)) / (\pi - d(y_1, y_2))$
    is close to $1$
    when $d(x_1, x_2)$ is close to $\pi$,
    which will allow us to prove that $d(x_1, x_2) = d(y_1, y_2)$
    with a little bit more effort.
\end{rem}

The following lemma shows that the inverse images of a pair of antipodal points can not be too close,
which can be used to prove Proposition \ref{prop_neq}.
\begin{lem}
    Suppose $p, p' \in \sphere^n$ are antipodal points,
    and suppose the Lipschitz constant of
    $f : \sphere^n \times \sphere^n \rightarrow \sphere^n$
    is $\sqrt{2}$,
    then
    \begin{align}
	N(f^{-1}(p), \frac{\pi}{\sqrt{2}} ) \bigcap f^{-1}(p') = \emptyset,
	\label{eq_far}
    \end{align}
    where
    \begin{align*}
	N(f^{-1}(p), \frac{\pi}{\sqrt{2}} ) = \{(x, y) \in \sphere^n \times \sphere^n : d((x, y), f^{-1}(p)) < \frac{\pi}{\sqrt{2}} \}
    \end{align*}
    is the $\frac{\pi}{\sqrt{2}}$-neighborhood of $f^{-1}(p)$.
    \label{lem_far}
\end{lem}

\begin{proof}
    For any $(x_1, y_1) \in f^{-1}(p)$ and
    $(x'_1, y'_1) \in f^{-1}(p')$,
    since
    \begin{align*}
	d(f(x_1, y_1), f(x'_1, y'_1)) \le \sqrt{2} d( (x_1, y_1), (x'_1, y'_1) ),
    \end{align*}
    \begin{align*}
	d( (x_1, y_1), (x'_1, y'_1) ) \ge \frac{d(f(x_1, y_1), f(x'_1, y'_1))}{\sqrt{2}} = \frac{\pi}{\sqrt{2}}.
    \end{align*}
\end{proof}

\begin{proof}[Proof of Proposition \ref{prop_neq}]
    Lemma \ref{lem_far} implies that
    the complement of $N(f^{-1}(p), \frac{\pi}{\sqrt{2}} )$
    contains $f^{-1}(p')$,
    which intersect cycles in the homology class of $\set{x} \times \sphere^n$.
    If $(x_1, y_1)$ and $(x_2, y_2)$ are in $f^{-1}(p)$ but they do not satisfy
    \eqref{eq_neq1} or \eqref{eq_neq2},
    then we can construct a cycle
    \begin{enumerate}
	\item within the same homology class as $\set{x} \times \sphere^n$ (and hence intersecting $f^{-1}(p')$),
	\item and lying in $N(f^{-1}(p), \frac{\pi}{\sqrt{2}} )$.
    \end{enumerate}
    This will contradict Lemma \ref{lem_far}.

    The cycle in $\sphere^n \times \sphere^n$ which we need to construct
    will be a topological sphere which contains $(x_1, y_1)$
    and $(x_2, y_2)$.
    Its projection to the first $\sphere^n$ will be
    a shortest geodesic from $x_1$ to $x_2$,
    and its projection to the second $\sphere^n$
    will be the full $\sphere^n$.

    \textbf{The first step in the construction:
    break $\sphere^n$ down to
a family of curves from $y_1$ to $y_2$}.
    The curves will be parametrized by the unit tangent vectors
    $U_{y_1} \sphere^n$.
    For each unit tangent vector $X \in U_{y_1} \sphere^n$,
    there is a unique $2$-plane spanned
    by $y_1$, $y_2$, and $X$ in $\real^{n+1}$.
    The intersection of this $2$-plane and $\sphere^n$
    will be a circle containing $y_1$ and $y_2$,
    thus two simple curves $\alpha_X$ and $\alpha_{-X}$ from $y_1$ to $y_2$,
    where the direction of $\alpha_X$ is the same as $X$,
    and where the direction of $\alpha_{-X}$ is opposite to $X$.
    (See Figure \ref{fig_curve}.)
    We can further specify that $\alpha_X : [0, 1] \rightarrow \sphere^n$
    is of constant speed.
    \begin{figure}[h]
	\includegraphics{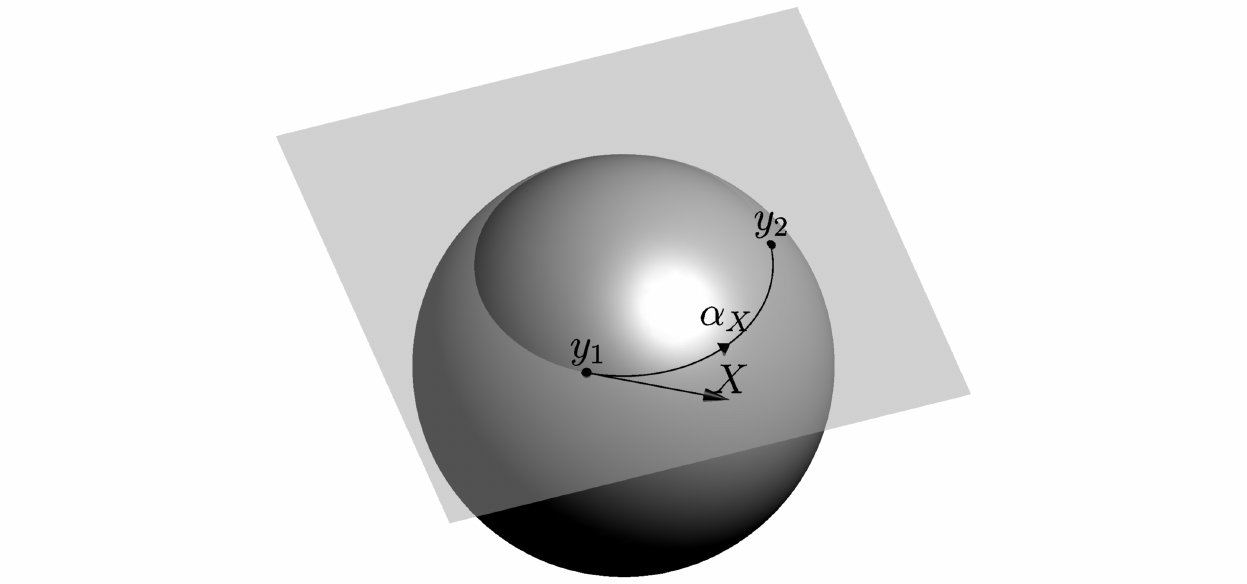}
	\caption{$\alpha_X : [0, 1] \rightarrow \sphere^n$ is an arc from $y_1$ to $y_2$ which is tangent to $X$.}
	\label{fig_curve}
    \end{figure}

    \textbf{The second step in the construction:}
    let $\beta : [0, 1] \rightarrow \sphere^n$ be a shortest geodesic
    from $x_1$ to $x_2$,
    and let $\gamma_X(t) = (\beta(t), \alpha_X(t))$,
    then the cycle we need is
    \begin{align*}
	S = \{\gamma_X(t) : X \in U_{y_1} \sphere^n, t \in [0, 1] \}.
	\label{}
    \end{align*}

    Now we can verify that $S$ has
    the desired homology,
    which will imply that
    $f^{-1}(p') \bigcap S \neq \emptyset$.
    When $x_1 = x_2$ and when $y_1$ and $y_2$ are antipodal points,
    $S$ is exactly $\set{x_1} \times \sphere^n$.
    If we move $x_2$ or $y_2$ continuously,
    $S$ is also deformed continuously.
    Hence for any $(x_1, y_1)$ and $(x_2, y_2)$,
    $S$ and $\set{x_1} \times \sphere^n$ are in the same homology class.
    Since $f$ and $m$ are homotopic,
    $f^{-1}(p')$ and $m^{-1}(p')$ have the same homology class.
    As $m^{-1}(p')$ and $\set{x_1} \times \sphere^n$ have exactly one intersection point $(x_1, x_1^{-1} p')$,
    \begin{align}
	f^{-1}(p') \bigcap S \neq \emptyset.
	\label{eq_int}
    \end{align}

    The last step is to estimate $\ell(\gamma_X)$, the length of $\gamma_X$.
    If the inequality \eqref{eq_neq2} is violated,
    then the estimate will imply that $S \subset N(f^{-1}(p), \frac{\pi}{\sqrt{2}} )$.
    Since $\ell(\gamma_X) = \sqrt{\left( \ell(\alpha_X) \right)^2 + \left( \ell(\beta) \right)^2}$, we need to first estimate $\ell(\alpha_X)$ and $\ell(\beta)$.
    By our construction,
    for any $X \in U_{y_1} \sphere^n$,
    $\ell(\alpha_X) \le 2 \pi - d(y_1, y_2)$
    and $\ell(\beta) = d(x_1, x_2)$.
    Hence
    \begin{align}
	\ell(\gamma_X) \le \sqrt{\left( 2 \pi - d(y_1, y_2) \right)^2 + \left( d(x_1, x_2) \right)^2}.
	\label{eq_est}
    \end{align}
    If \eqref{eq_neq2} is violated,
    then \eqref{eq_est} implies $\ell(\gamma_X) < \sqrt{2} \pi$,
    and hence
    \begin{align}
	\gamma_X \subset N(\{(x_1, y_1), (x_2, y_2)\}, \frac{\pi}{\sqrt{2}} ) \subset N(f^{-1}(p), \frac{\pi}{\sqrt{2}} ).
	\label{eq_short}
    \end{align}
    As \eqref{eq_short} is true for any $\gamma_X$,
    we have
    \begin{align}
	S \subset N(f^{-1}(p), \frac{\pi}{\sqrt{2}} ).
	\label{eq_small}
    \end{align}
    Now \eqref{eq_int} and \eqref{eq_small} implies
    \begin{align}
	N(f^{-1}(p), \frac{\pi}{\sqrt{2}} ) \bigcap f^{-1}(p') \neq \emptyset,
	\label{}
    \end{align}
    which contradicts \eqref{eq_far}.
    Therefore \eqref{eq_neq2} holds.
    The proof for \eqref{eq_neq1} is similar.
    This completes the proof of Proposition \ref{prop_neq}.
\end{proof}

\subsection{Fibers of $f$ are spheres}
Here are some observations of Figure \ref{fig_neq},
and we shall solidify these ideas to prove that the fibers of $f$ are the graphs of isometries of $\sphere^n$.
As can be seen from Figure \ref{fig_neq},
$y_2$ is the antipodal point of $y_1$
if $x_2$ is the antipodal point of $x_1$.
In particular,
there is only one $y_2 \in \sphere^n$ such that $(x_2, y_2) \in f^{-1}(p)$.
Moreover,
$\pi - d(x_1, x_2)$ and $\pi - d(y_1, y_2)$ are roughly equal when $d(x_1, x_2)$ is close to $\pi$.
These observations leads to the proof of the following proposition.

\begin{prop}
    $f^{-1}(p)$ is the graph of a isometry $h_p : \sphere^n \rightarrow \sphere^n$ where $h_p(x) = y$ if and only if $(x, y) \in f^{-1}(p)$.
    \label{prop_iso}
\end{prop}

\begin{proof}
    $h_p$ is well defined if
    \begin{enumerate}
	\item for any $x \in \sphere^n$ there is $y \in \sphere^n$ such that $f(x, y) = p$,
	\item $f(x, y) = f(x, y'') = p$ implies $y = y''$.
    \end{enumerate}
    For any $x_1 \in \sphere^n$,
    since $m|_{\set{x_1} \times \sphere^n}$ is an isometry,
    $f|_{\set{x_1} \times \sphere^n}$ is also surjective.
    Hence there is $y_1 \in \sphere^n$ such that $f(x_1, y_1) = p$.

    Now let $x_1' \in \sphere^n$ be the antipodal point of $x_1$,
    then there is $y_1' \in \sphere^n$ such that $(x_1', y_1') \in f^{-1}(p)$.
    Since $d(x_1, x_1') = \pi$,
    \eqref{eq_neq1} implies $d(y_1, y_1') \ge \pi$,
    that is,
    $y_1'$ is the antipodal point of $y_1$.

    Suppose $(x_1, y_1'') \in f^{-1}(p)$.
    Since we also have $(x_1', y_1') \in f^{-1}(p)$,
    then $d(x_1, x'_1) = \pi$ and \eqref{eq_neq1} imply $d(y_1'', y_1') \ge \pi$,
    that is,
    $y_1''$ is the antipodal point of $y_1'$
    and hence $y_1'' = y_1$.
    Therefore $h_p$ is well defined.

    Similarly,
    we can define $k_p : \sphere^n \rightarrow \sphere^n$ as $k_p(y_1) = x_1$
    if and only if we have $(x_1, y_1) \in f^{-1}(p)$.
    Then $h_p \circ k_p$ and $k_p \circ h_p$ are identity maps,
    and hence $h_p$ is a bijection.

    Next, we shall prove that $L_{h_p}(\sphere^n) \le 1$.
    Since $(\sphere^n, d)$ is a length space,
    it suffices to show that the local Lipschitz constant of $h_p$ is $\le 1$ \cite{gromov-2007},
    i.e.,
    \begin{align}
	\limsup_{x_2 \rightarrow x_1} \frac{d(h(x_1), h(x_2))}{d(x_1, x_2)} \le 1.
	\label{eq_local_lip}
    \end{align}
    For any $(x_1, y_1), (x_2, y_2) \in f^{-1}(p)$,
    \eqref{eq_neq1} implies
    \begin{align*}
	\left(2 \pi - d(x'_1, x_2)\right)^2 + \left( d(y'_1, y_2) \right)^2 \ge 2 \pi^2,
	\label{}
    \end{align*}
    where $x'_1$ and $x'_2$ are the antipodal points of $x_1$ and $x_2$ respectively.
    Since $d(x'_1, x_2) = \pi - d(x_1, x_2)$
    and since $d(y'_1, y_2) = \pi - d(y_1, y_2)$,
    \begin{align*}
	\left(\pi + d(x_1, x_2)\right)^2 + \left(\pi - d(y_1, y_2) \right)^2 \ge 2 \pi^2.
	\label{}
    \end{align*}
    Hence
    \begin{align*}
	d(y_1, y_2)
	&\le \pi - \sqrt{2 \pi^2 - (\pi + d(x_1, x_2))^2}\\
	&= \pi - \sqrt{\pi^2 - 2 \pi d(x_1, x_2) + o(d(x_1, x_2))}\\
	&= d(x_1, x_2) + o(d(x_1, x_2)),
	\label{}
    \end{align*}
    which implies \eqref{eq_local_lip}.

    Similarly,
    $h_p^{-1} = k_p$ is also of Lipschitz constant at most $1$,
    and hence $h_p$ is an isometry.
\end{proof}

\subsection{Fibers of $f$ are parallel}
By now,
we know that the fibers of $f$ are graphs of isometries.
As stated before,
these fibers are also parallel.

\begin{prop}
    Fibers of $f$ are parallel. In other words,
    for any $p_1, p_2 \in \sphere^n$
    and $(x_1, y_1) \in f^{-1}(p_1)$,
    \begin{align}
	d( (x_1, y_1), f^{-1}(p_2)) = d(f^{-1}(p_1), f^{-1}(p_2)) = \frac{d(p_1, p_2)}{\sqrt{2}}.
	\label{eq_parallel}
    \end{align}
    \label{prop_parallel}
\end{prop}

\begin{proof}
    By symmetry,
    we will assume $p_2 = 1$ without loss of generality.
    (Here we view points in $\sphere^n$ as unit quaternions.)

    Moreover,
    $f^{-1}(1)$ is the graph of an isometry $h_1 : \sphere^n \rightarrow \sphere^n$,
    by Proposition \ref{prop_iso}.
    Define an isometry $H_1 : \sphere^n \times \sphere^n \rightarrow \sphere^n \times \sphere^n$
    as $H_1(x, y) = (x^{-1}, h_1^{-1}(y))$,
    then $H_1(f^{-1}(1)) = m^{-1}(1)$.
    So we will assume $f^{-1}(1) = m^{-1}(1)$ without loss of generality.

    When $f = m$, \eqref{eq_parallel} becomes
    \begin{align}
	d( (x_1, y_1), m^{-1}(1)) = \frac{d(p_1, 1)}{\sqrt{2}}.
	\label{eq_special}
    \end{align}
    Since $m$ extends to the (scaled) Hopf fibrations $\tilde{m} : \sqrt{2} \sphere^{2n+1} \rightarrow \sphere^{n+1}$ defined as
    \begin{align}
	\tilde{m}(z_1, z_2) = (z_1 z_2, \frac{|z_1|^2 - |z_2|^2}{2}),
	\label{eq_hopf}
    \end{align}
    \eqref{eq_special} follow easilly from \cite{gluck-warner-ziller-1986}.

    Next, we shall use this special case to prove this proposition.
    \eqref{eq_special} implies that
    \begin{align*}
	N(m^{-1}(1), \frac{\pi}{\sqrt{2}})
	= m^{-1}(N(\{1\}, \pi)) = m^{-1}(\sphere^n \setminus \{-1\}),
    \end{align*}
    and thus
    \begin{align*}
	N(f^{-1}(1), \frac{\pi}{\sqrt{2}}) = m^{-1}(\sphere^n \setminus \{-1\}).
    \end{align*}
    Taking $p = 1$ in Lemma \ref{lem_far},
    we have
    $f^{-1}(-1) \bigcap N(f^{-1}(1), \frac{\pi}{2}) = \emptyset$,
    and hence
    \begin{align*}
	f^{-1}(-1) &\subset \sphere^3 \times \sphere^3 \setminus N(f^{-1}(1), \frac{\pi}{2})\\
	&= \sphere^3 \times \sphere^3 \setminus m^{-1}(\sphere^3 \setminus \{-1\})\\
	&= m^{-1}(-1).
    \end{align*}
    Recall that $f^{-1}$ is the graph of an isometry,
    so it can only be $m^{-1}(-1)$.
    This proves \eqref{eq_parallel} with $p_1 = -1$ and $p_2 = 1$,
    and it remains to verify that other fibers are also parallel to $f^{-1}(1)$.

    Suppose any $p_1 \in \sphere^n$
    and $(x_1, y_1) \subset f^{-1}(p_1)$.
    Since $f$ is of Lipschitz constant $\sqrt{2}$,
    \begin{align}
	d( (x_1, y_1), f^{-1}(1) ) \ge \frac{d(p_1, 1)}{\sqrt{2}},
	\label{eq_ge1}
    \end{align}
    and
    \begin{align}
	d( (x_1, y_1), f^{-1}(-1) ) \ge \frac{d(p_1, -1)}{\sqrt{2}}.
	\label{eq_ge2}
    \end{align}
    On the other hand,
    since $(x_1, y_1) \in M^{-1}(x_1 y_1)$ and $f^{-1}(1) = m^{-1}(1)$,
    \eqref{eq_special} implies
    \begin{align}
	d( (x_1, y_1), f^{-1}(1) ) = \frac{d(x_1 y_1, 1)}{\sqrt{2}},
	\label{eq_eq1}
    \end{align}
    and similarly
    \begin{align}
	d( (x_1, y_1), f^{-1}(-1) ) = \frac{d(x_1 y_1, -1)}{\sqrt{2}}.
	\label{eq_eq2}
    \end{align}
    By \eqref{eq_ge1}, \eqref{eq_ge2}, \eqref{eq_eq1} and \eqref{eq_eq2},
    \begin{align*}
	\frac{d(-1, 1)}{\sqrt{2}}
	&= \frac{d(x_1 y_1, 1)}{\sqrt{2}} + \frac{d(x_1 y_1, -1)}{\sqrt{2}}\\
	&= d( (x_1, y_1), f^{-1}(1) ) + d( (x_1, y_1), f^{-1}(-1) )\\
	&\ge \frac{d(p_1, 1)}{\sqrt{2}} + \frac{d(p_1, -1)}{\sqrt{2}}\\
	&= \frac{d(-1, 1)}{\sqrt{2}}.
    \end{align*}
    Thus all the inequalities in the above equations should be equalities,
    so 
    \begin{align*}
	d( (x_1, y_1), f^{-1}(1) ) = \frac{d(p_1, 1)}{\sqrt{2}}.
    \end{align*}
    This completes the proof of Proposition \ref{prop_parallel}.
\end{proof}

\subsection{Proof of the uniqueness}
We can embed $\sphere^n \times \sphere^n$ into $\sqrt{2} \sphere^{2n + 1} \subset \real^{2n + 2} = \real^{n+1} \times \real^{n+1}$ by embeding each $\sphere^n$ into $\real^{n+1}$.
Proposition \ref{prop_iso} implies that
every fiber of $f$ lies in a \\$(n+1)$-plane in $\real^{2n+2}$,
so these fibers are great $n$-spheres in $\sqrt{2} \sphere^{2n+1}$.
Also, Proposition \ref{prop_parallel} implies that
these fibers (which are $n$-spheres) are parallel in $\sqrt{2} \sphere^{2n+1}$.
Now the following result from \cite{wolf-1963, wong-1961}
shows that $f$ and $m$ extend to isometric fibrations on $\sqrt{2} \sphere^{2n+1}$.

\begin{prop} [\cite{wolf-1963, wong-1961}]
    Any fibration $f$ of $\sphere^{n} \times \sphere^{n}$ ($n = 1, 3, 7$) by parallel great $n$-spheres
    extends to a parallel fibration $\tilde{f}$ of all of $\sqrt{2} \sphere^{2n + 1}$ by parallel great $n$-spheres such that the following diagram commute,
    \begin{align*}
	\begin{CD}
	    \sphere^n \times \sphere^n @>\text{inclusion}>> \sqrt{2}\sphere^{2n+1} @>g_1>> \sqrt{2}\sphere^{2n+1}\\
	    @VVfV @VV\tilde{f}V @VV\tilde{m}V\\
	    \sphere^n @>e>> \sphere^{n+1} @>g_2>> \sphere^{n+1}
	\end{CD}
    \end{align*}
    where $e$ is a map, $g_1$ and $g_2$ are isometries, and $\tilde{m}$ is the Hopf fibration defined in \eqref{eq_hopf}.
    \label{prop_ext}
\end{prop}
\begin{proof}
    Notice that parallel great $n$-spheres in $\sphere^{2n+1}$ span isoclinic $(n+1)$-planes in $\real^{2n + 2}$.

    \cite[Theorem 7]{wolf-1963} states that any $n$-dimensional ($n = 1, 3, 7$) family of isoclinic $(n+1)$-planes in $\real^{2n + 2}$ can be extend to an $(n+1)$-dimensional maximal family of isoclinic $(n + 1)$-planes in $\real^{2n + 2}$.
    Also, all $(n+1)$-dimensional maximal families of isoclinic $(n+1)$-planes in $\real^{2n + 2}$ are isometric to each other by the same theorem.
    Thus there is always a map $\tilde{f} : \sqrt{2} \sphere^{2n+1} \rightarrow \sphere^{n+1}$ isometric to $\tilde{m}$ such that any fiber of $f$ is also a fiber of $\tilde{f}$.

    Finally, define the map $e$ as $e(f(x)) = \tilde{f}(x)$. To check that the map is well-defined, let $f(x) = f(y)$, then $x$ and $y$ are in a same fiber of $f$ and thus in a same fiber of $\tilde{f}$, which implies that $\tilde{f}(x) = \tilde{f}(y)$. So $e$ is well-defined.
\end{proof}

Now we can finish the proof using an argument due to Herman Gluck.

\begin{proof}
    [Proof of Theorem \ref{dim3}]
    We need to prove that any fibration $f$ of $\sphere^{n} \times \sphere^{n}$ by parallel great $n$-spheres
    is isometric to $m$.

    Let $i : \sphere^{n} \rightarrow \sphere^{n+1}$ be the inclusion map
    defined as $i(x) = (x, 0)$.
    Then $i \circ m$ extends to the Hopf fibration $\tilde{m} : \sqrt{2} \sphere^{2n + 1} \rightarrow \sphere^{n+1}$.

    Extend $f$ to $\tilde{f}$ and obtain $e$, $g_1$ and $g_2$ as in Proposition \ref{prop_ext}.
    For any $x, y \in \sphere^n$,
    \begin{align*}
	d(x, y)
	&= \frac{1}{\sqrt{2}} d(f^{-1}(x), f^{-1}(y))\\
	&= d(\tilde{f}(f^{-1}(x)), \tilde{f}(f^{-1}(y)))\\
	&= d(e(x), e(y)),
    \end{align*}
    so $e$ is actually an isometric embedding.

    Pick an isometry $g_3 : \sphere^{n+1} \rightarrow \sphere^{n+1}$ homotopic to the identity map such that
    \begin{align*}
	g_3 \circ g_2 \circ e (\sphere^n) = i(\sphere^n).
    \end{align*}
    By the homotopy lifting property of the fibration $\tilde{m}$,
    there is an isometry\\$g_4 : \sqrt{2}\sphere^{2n+1} \rightarrow \sqrt{2} \sphere^{2n+1}$
    such that the following diagram commute.
    \begin{align*}
	\begin{CD}
	    \sqrt{2} \sphere^{2n+1} @>g_4>> \sqrt{2} \sphere^{2n+1}\\
	    @V\tilde{m}VV @V\tilde{m}VV\\
	    \sphere^{n+1} @>g_3>> \sphere^{n+1}
	\end{CD}
    \end{align*}
    (One can lift any curve in $\sphere^{n+1}$ to a horizontal curve in $\sqrt{2}\sphere^{2n+1}$, thus $g_3$, being homotopic to the identity map, can be lifted to $g_4$.)

    Since 
    \begin{align*}
	g_4 \circ g_1 (\sphere^n \times \sphere^n)
	&= \tilde{m}^{-1} \left(g_3 \circ g_2 \circ e (\sphere^n)\right)\\
	&= \tilde{m}^{-1} \left(i(\sphere^n)\right)\\
	&= \sphere^{n} \times \sphere^n,
    \end{align*}
    we can define an isometry $g_5 : \sphere^n \times \sphere^n \rightarrow \sphere^n \times \sphere^n$ as $g_5 = g_4 \circ g_1$ and
    an isometry $g_6 : \sphere^n \rightarrow \sphere^n$,
    as $g_6 = g_3 \circ g_2 \circ e$.
    Then $m \circ g_5 = g_6 \circ f$,
    that is,
    $m$ is isometric $f$.
    This completes the proof of Theorem \ref{dim3}.
\end{proof}

\section{An interesting counterexample}
\label{counter_example}
The authors of \cite{deturk-gluck-storm-2010} suspected that any Riemannian submersion between compact homogeneous spaces is a Lipschitz constant minimizer in its homotopy class, unique up to isometries.
The following example shows that this is not necessarily true if we drop the assumption on homogeneity,
even in the case where the receiving space is a circle.

Let $r(x) = 2 - \cos(4\pi x)$.
Define a Riemannian metric $g$ on the two torus $\torus = \real^2 / \integer^2$ as
$g_{11} = 1$, $g_{12} = g_{21} = 0$ and $g_{22}(x, y) = (r(x))^2$.
In otherwords,
this torus is the quotient of a surface of revolution.
For any $a \in [0, 1)$,
define a family of closed curves $\gamma_a(t) = (x(t), y(t))$ in the two torus as
\begin{align*}
    \begin{cases}
        \gamma_a(0) = (0, a),\\
        \frac{dx}{dt} = \frac{1}{r(x(t))},\\
	\frac{dy}{dt} = \frac{1}{r(x(t))}\sqrt{1 - \frac{1}{[r(x(t))]^2}} & \text{if $0 \le x(t) \le \frac{1}{2}$},\\
	\frac{dy}{dt} = -\frac{1}{r(x(t))}\sqrt{1 - \frac{1}{\left[r(x(t))\right]^2}} & \text{if $\frac{1}{2} \le x(t) \le 1$},\\
    \end{cases}
    \label{}
\end{align*}
and define $f : \torus \rightarrow \real / \integer$ defined as $f(\gamma_a(t)) = a$, where $\real / \integer$ is a circle of length $1$ with the standard metric.
(See Figure \ref{fig_f3d}.)
We shall prove that $f$ is a Riemannian submersion in the next paragraph.
However,
the map $g : \torus \rightarrow \real / \integer$ defined as $g(x, y) = y$ is homotopic to $f$
and is also of Lipschitz constant $1$, but it is not a Riemannian submersion;
in other words, $f$ is not the unique Lipschitz constant minimizer even up to isometries.

\begin{figure}[h]
    \centering
    \includegraphics[width=\textwidth]{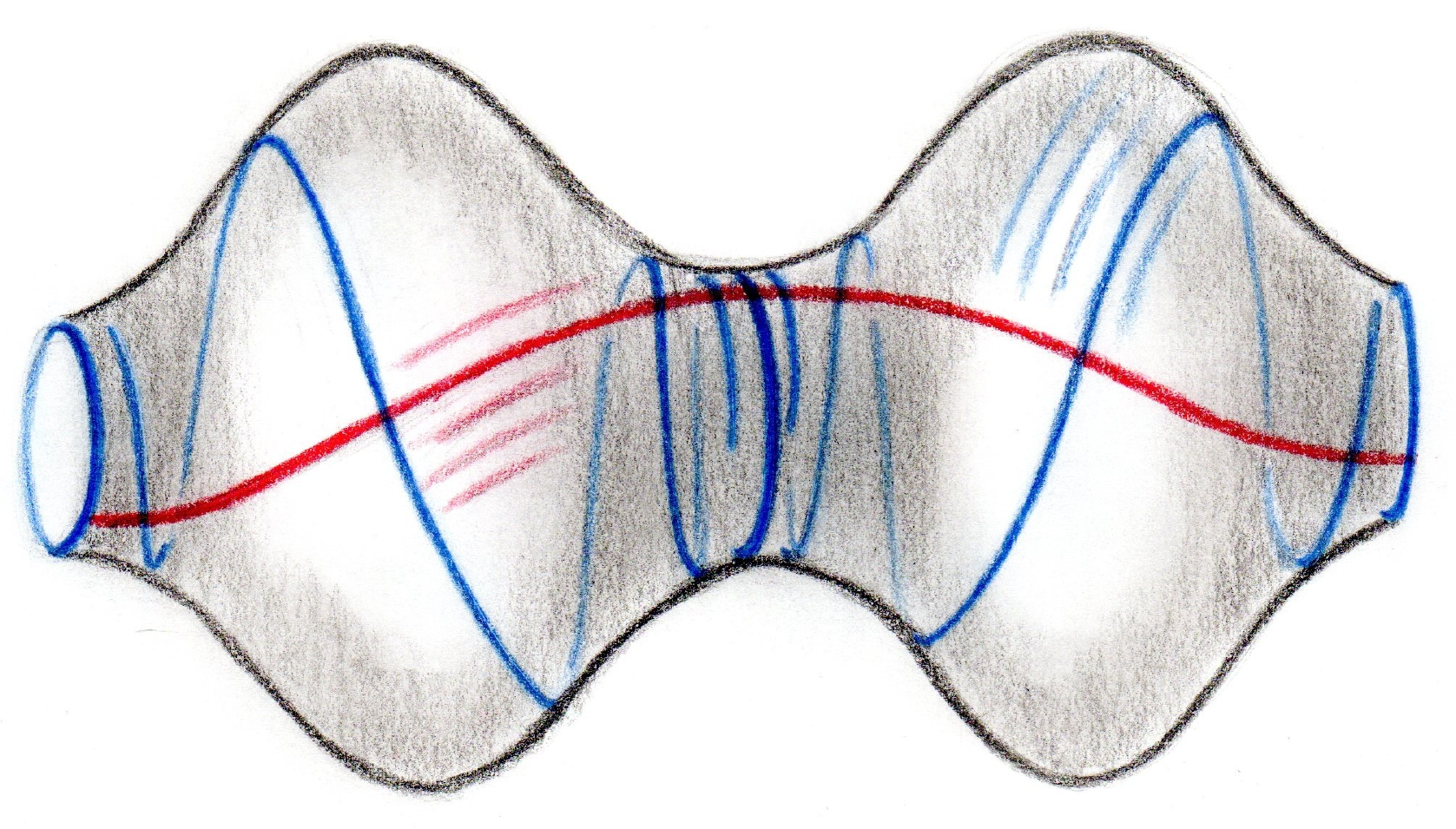}
    \caption{The torus in the example is depicted here, with left and right circle identified. The red curves are the level sets $\gamma_a$ of $f$. Each is a simple closed curve which goes around the torus the ``long way''. The blue curves are their orthogonal trajectories, and thus are the integral curves of $\nabla f$. Two of them are simple closed curves which go around the torus the ``short way''. The remaining blue curves are geodesics winding around the torus infinitely often the short way, and limit on the two closed ones.}
\label{fig_f3d}
\end{figure}

We can verify that $f$ is a Riemannian submersion as following.
Without loss of generality,
assume $0 \le x(t) \le \frac{1}{2}$,
then
\begin{align*}
    \gamma_a'(t) = \frac{1}{r(x(t))} \frac{\partial}{\partial x} + \frac{1}{r(x(t))}\sqrt{1 - \frac{1}{[r(x(t))]^2}} \frac{\partial}{\partial y},
    \label{}
\end{align*}
and thus
\begin{align}
    \frac{1}{r(x(t))} \frac{\partial f}{\partial x} + \frac{1}{r(x(t))}\sqrt{1 - \frac{1}{[r(x(t))]^2}} \frac{\partial f}{\partial y} = 0
    \label{eq_ct0}
\end{align}
as $\gamma_a'(t)$ is tangent to fibers.
By symmetry,
$f$ maps the circle $\set{x} \times \real / \integer$ to $\real / \integer$ uniformly,
and hence
\begin{align}
    \frac{\partial f}{\partial y} = 1.
    \label{eq_cty}
\end{align}
\eqref{eq_ct0} and \eqref{eq_cty} imply that
\begin{align}
    \frac{\partial f}{\partial x} = -\sqrt{1 - \frac{1}{(r(x))^2}}.
    \label{eq_ctx}
\end{align}
Now \eqref{eq_cty} and \eqref{eq_ctx} imply that
\begin{align}
    \left(-\sqrt{1 - \frac{1}{(r(x))^2}} \frac{\partial}{\partial x} + \frac{1}{(r(x))^2} \frac{\partial}{\partial y} \right) f = 1,
    \label{}
\end{align}
where $-\sqrt{1 - \frac{1}{(r(x))^2}} \frac{\partial}{\partial x} + \frac{1}{(r(x))^2} \frac{\partial}{\partial y}$ is a unit normal vector of a fiber.
Therefore $f$ is a Riemannian submersion.


%
%
%
%
%
%
%
%

\bibliography{mybib}{}

\begin{thebibliography}{GWZ86}

\bibitem[DGS]{deturk-gluck-storm-2010}
D.~{DeTurck}, H.~{Gluck}, and P.~{Storm}.
\newblock {Lipschitz minimality of Hopf fibrations and Hopf vector fields}.
\newblock \href {http://arxiv.org/abs/1009.5439} {\path{arXiv:1009.5439}}.

\bibitem[Gro78]{gromov-1978}
M.~Gromov.
\newblock Homotopical effects of dilatation.
\newblock {\em J. Differential Geom.}, 13(3):303--310, 1978.

\bibitem[Gro99]{gromov-2007}
M.~Gromov.
\newblock {\em Metric structures for {R}iemannian and non-{R}iemannian spaces},
  volume 152 of {\em Progress in Mathematics}.
\newblock Birkh\"auser Boston Inc., Boston, MA, 1999.

\bibitem[GWZ86]{gluck-warner-ziller-1986}
H.~Gluck, F.~Warner, and W.~Ziller.
\newblock The geometry of the {H}opf fibrations.
\newblock {\em Enseign. Math. (2)}, 32(3-4):173--198, 1986.

\bibitem[Oli66]{olivier-1966}
R.~Olivier.
\newblock \"{U}ber die {D}ehnung von {S}ph\"arenbbildungen.
\newblock {\em Invent. Math.}, 1:380--390, 1966.

\bibitem[Whi86]{white-1986}
B.~White.
\newblock Infima of energy functionals in homotopy classes of mappings.
\newblock {\em J. Differential Geom.}, 23(2):127--142, 1986.

\bibitem[Wol63]{wolf-1963}
J.A. Wolf.
\newblock Geodesic spheres in {G}rassmann manifolds.
\newblock {\em Illinois J. Math.}, 7:425--446, 1963.

\bibitem[Won61]{wong-1961}
Y.-C. Wong.
\newblock Isoclinic {$n$}-planes in {E}uclidean {$2n$}-space, {C}lifford
  parallels in elliptic {$(2n-1)$}-space, and the {H}urwitz matrix equations.
\newblock {\em Mem. Amer. Math. Soc. No.}, 41:iii+112, 1961.

\end{thebibliography}
\bibliographystyle{alphaurl}
\end{document}